\documentclass[11pt]{article}

\usepackage{amsfonts,amssymb,amsmath,amsthm,latexsym,color,epsfig}
\usepackage[textsize=scriptsize,colorinlistoftodos]{todonotes}
\theoremstyle{plain}
\newtheorem{thm}{Theorem}[section]

\newtheorem{lem}[thm]{Lemma}

\newtheorem{conj}[thm]{Conjecture}

\title{\vspace{-0.7cm}Books versus triangles at the extremal density}
\author{David Conlon\thanks{Department of Mathematics, California Institute of Technology, Pasadena, CA 91125, USA. Email: {\tt dconlon@caltech.edu}. Research supported by ERC Starting Grant 676632.}
\and
Jacob Fox\thanks{Department of Mathematics, Stanford University, Stanford,
CA 94305, USA. Email: {\tt jacobfox@stanford.edu}. Research supported by
a Packard Fellowship and by NSF Career Award DMS-1352121.}
\and
Benny Sudakov\thanks{Department of Mathematics, ETH, 8092 Zurich, Switzerland. Email: {\tt benjamin.sudakov@math.ethz.ch}. Research supported in part by SNSF grant 200021-175573.}}
\date{}

\begin{document}
\maketitle

\begin{abstract} 
A celebrated result of Mantel shows that every graph on $n$ vertices with $\lfloor n^2/4 \rfloor + 1$ edges must contain a triangle. A robust version of this result, due to Rademacher, says that there must in fact be at least $\lfloor n/2 \rfloor$ triangles in any such graph. Another strengthening, due to the combined efforts of many authors starting with Erd\H{o}s, says that any such graph must have an edge which is contained in at least $n/6$ triangles. Following Mubayi, we study the interplay between these two results, that is, between the number of triangles in such graphs and their book number, the largest number of triangles sharing an edge. Among other results, Mubayi showed that for any $1/6 \leq \beta < 1/4$ there is $\gamma > 0$ such that any graph on $n$ vertices with at least $\lfloor n^2/4\rfloor + 1$ edges and book number at most $\beta n$ contains at least $(\gamma -o(1))n^3$ triangles. He also asked for a more precise estimate for $\gamma$ in terms of $\beta$. We make a conjecture about this dependency and prove this conjecture for $\beta = 1/6$ and for $0.2495 \leq \beta < 1/4$, thereby answering Mubayi's question in these ranges.
\end{abstract}

\section{Introduction}

Mantel's theorem \cite{Man} from 1907 is among the earliest results in extremal graph theory. It states that the maximum number of edges that a triangle-free graph on $n$ vertices can have is $\lfloor n^2/4 \rfloor$, with equality if and only if the graph is the balanced complete bipartite graph. So a graph on $n$ vertices with one more edge must have at least one triangle. Must it have many triangles? Must there be an edge in many triangles? Such questions have a long history of study in extremal graph theory.

In unpublished work, Rademacher answered the first question above in 1950, proving that every graph on $n$ vertices with $\lfloor n^2/4 \rfloor+1$ edges has at least $\lfloor n/2 \rfloor$ triangles, which is tight by adding an edge inside the larger part of a balanced complete bipartite graph.  Erd\H{o}s \cite{Er62} then extended this result to graphs with a linear number of extra edges and, in \cite{Er62b}, studied the problem for larger cliques. Over the last fifty years, many further results in this direction have been obtained by various researchers, see, e.g., \cite{Bol, LPS, Lov-Sim, Nikiforov, Razborov, Reiher} and their references.  

The second question, about finding an edge in many triangles, was first studied by Erd\H{o}s \cite{Er62} in 1962. A {\it book} in a graph is a collection of triangles that have an edge in common. The \emph{size} of the book is the number of such triangles. The \emph{book number} of a graph $G$, denoted by $b(G)$, is the size of the largest book in the graph. Erd\H{o}s proved that every graph $G$ on $n$ vertices with $\lfloor n^2/4 \rfloor+1$ edges satisfies $b(G) \geq n/6-O(1)$ and conjectured that the $O(1)$-term can be removed. Solving this  conjecture  and answering the second question above, Edwards and, independently, Khad\v{z}iivanov and Nikiforov~\cite{KN79} proved that every such graph satisfies $b(G)\geq n/6$, which is tight. 

Our concern here is with a problem of Mubayi \cite{Mubayi} about the interplay between the two questions above. More precisely, if a graph $G$ on $n$ vertices with $\lfloor n^2/4 \rfloor+1$ edges satisfies $b(G) \leq b$, at least how many triangles must it have? We write $t(n, b)$ for this minimum number. Mubayi proved that for fixed $\beta \in (1/4,1/2)$, if $b(G)<\beta n$, then $t(n,b) \geq \left(\frac{1}{2}\beta(1-2\beta)-o(1)\right)n^2$, a bound which is asymptotically tight. He also showed that $t(n, b)$ changes from quadratic to cubic in $n$ when $b \approx n/4$. More precisely, he proved that for each $\beta \in (1/6,1/4)$ there is $\gamma>0$ such that $t(n,\beta n) \geq \gamma n^3$. He then asked for a more precise determination of the optimal $\gamma$ in terms of $\beta$, but added that the problem `seems very hard'. Our contribution in this paper is to make a conjecture about this dependency and to confirm this conjecture for $\beta = 1/6$ and for $0.2495 \leq \beta < 1/4$.  

To say more, consider the $3$-prism graph, the skeleton of the $3$-prism, consisting of two disjoint triangles with a perfect matching between them. For nonnegative integers $b$ and $n$ with $b \leq n/4$, let $S_{b,n}$ be the graph on $n$ vertices formed by blowing up the $3$-prism graph, where four of the six parts, corresponding to the vertices of two edges of the matching, are of size $b$, and the remaining two parts are of size $\lfloor (n-4b)/2\rfloor$ and $\lceil (n-4b)/2\rceil$. Restated, $S_{b,n}$ has vertex set consisting of six parts $U_1,U_2,U_3,V_1,V_2,V_3$ with $|U_1|=|U_2|=|V_1|=|V_2|=b$, $|U_3|=\lfloor (n-4b)/2\rfloor$, $|V_3|=\lceil (n-4b)/2\rceil$, $U_i$ is complete to $U_j$ for $i \not = j$, $V_i$ is complete to $V_j$ for $i \not = j$, $U_i$ is complete to $V_i$ for each $i$ and there are no other edges. The graph $S_{b,n}$ has $n$ vertices, $\lfloor n^2/4 \rfloor$ edges, book number $b$ if $b \geq n/6$ and $b^2(n-4b)$ triangles. If $b=0$ or $n/4$, then $S_{b,n}$ is the balanced complete bipartite graph, but otherwise has triangles. We make the following conjecture.\footnote{Though the final version of Mubayi's paper~\cite{Mubayi} contains no conjecture about the behaviour of $t(n, \beta n)$ for $1/6 \leq \beta < 1/4$, the original arXiv version described a construction which is almost identical to that given here. As such, we might well ascribe an approximate version of Conjecture~\ref{mainconj} to him.}

\begin{conj} \label{mainconj}
If $n/6 \leq b < n/4$, then every graph on $n$ vertices with at least $\lfloor n^2/4 \rfloor$ edges and book number at most $b$ 
which is not the balanced complete bipartite graph has at least $b^2(n-4b)$ triangles, with equality if and only if the graph is $S_{b,n}$.  
\end{conj}

Our main result is a proof of Conjecture \ref{mainconj} when $b$ is not much smaller than $n/4$. 

\begin{thm}\label{thmbooknearquarter}
Conjecture \ref{mainconj} holds for graphs with at least $n^2/4$ edges if $0.2495n \leq b < n/4$.
\end{thm}

While Theorem \ref{thmbooknearquarter} is stated for graphs with at least $n^2/4$ edges, the proof is robust enough to yield the analogous result for graphs with $\lfloor n^2/4 \rfloor$ edges. We only prove the weaker statement for simplicity of presentation.  

The perceptive reader will have noticed that our Conjecture~\ref{mainconj} differs from Mubayi's question in one small, but important, point of detail: we allow our graphs to have $\lfloor n^2/4 \rfloor$ edges, whereas Mubayi looks at graphs with at least $\lfloor n^2/4 \rfloor + 1$ edges, thus guaranteeing that there are always some triangles. However, Conjecture~\ref{mainconj} also implies an asymptotically tight bound on the function $t(n, b)$. To see this, consider a slightly different blow-up of the $3$-prism graph, adding one vertex to each $U_i$ and subtracting one vertex from each $V_i$. If $n$ is even, we get a graph with book number $b+1$ and with three more edges and $n$ more triangles than $S_{n,b}$. If $n$ is odd, we get a graph with book number $b+1$ and with two more edges and $n-2b$ more triangles than $S_{n,b}$. We now delete two edges, each in $b+1$ triangles but not in a common triangle, if $n$ is even and one edge in $b+1$ triangles if $n$ is odd, yielding the bounds $t(n,b+1) \leq b^2(n-4b)+n-2(b+1)$ if $n$ is even and $t(n,b+1) \leq b^2(n-4b)+n-2b-(b+1)$ if $n$ is odd. Together with Conjecture~\ref{mainconj}, these constructions imply the required asymptotic estimate on $t(n,b)$ for $n/6 \leq b \leq n/4-\omega(1)$, where the $\omega(1)$ term indicates any function tending to infinity with $n$. 

We also study what happens at the other end of the range, showing that Conjecture~\ref{mainconj} holds for $b = n/6$. More precisely, we will make use of results from a paper of Bollob\'as and Nikiforov~\cite{BN05}, themselves derived from the earlier work of Edwards and Khad\v{z}iivanov--Nikiforov~\cite{KN79}, to show that the conjecture holds in this case.

\begin{thm}\label{thmbooknearsixth}
Conjecture~\ref{mainconj} holds for graphs with at least $n^2/4$ edges if $b = n/6$.
\end{thm}

Once again, the theorem holds for graphs with $\lfloor n^2/4 \rfloor$ edges, but it is more convenient, principally from a notational standpoint, to assume that there are at least $n^2/4$ edges.

\vspace{3mm}

{\bf Notation.} For a graph $G$ and vertex $v$, the neighborhood $N(v)$ denotes the set of vertices adjacent to $v$, while the degree of $v$ is denoted by $d(v):=|N(v)|$ and the degree of $v$ into a vertex subset $A$ is denoted by $d_A(v):=|N(v) \cap A|$. For two vertices $u$ and $v$, their common neighborhood is denoted by $N(u,v)$ and their codegree $d(u,v)=|N(u,v)|$ is the number of vertices adjacent to both $u$ and $v$. The codegree of $u$ and $v$ into a vertex subset $A$ is denoted by $d_A(u,v):=|N(u,v) \cap A|$. For a vertex subset $A$, we write $E(A)$ for the set of edges in $A$ and $e(A)$ for the number of such edges. Similarly, for vertex subsets $A$ and $B$, the set of edges with one vertex in $A$ and the other in $B$ is denoted by $E(A,B)$ and the number of such edges is $e(A,B)=|E(A,B)|$. If the underlying graph $G$ is not clear from context, we include it in the notation.  

\section{Proof of Theorem~\ref{thmbooknearquarter}}

The following lemma gives a bound on the maximum cut of a graph with few triangles. The result and proof in the special case of triangle-free graphs is due to Erd\H{o}s, Faudree, Pach and Spencer \cite{EFPS}. 

\begin{lem}\label{firstlemma}
If $G$ is a graph with $n$ vertices, $m$ edges and $t$ triangles, then $G$ can be made bipartite by deleting 
at most $m-\frac{4m^2}{n^2}+\frac{6t}{n}$ edges. 
\end{lem} 

\begin{proof}
We will show that there is a vertex $x$ for which $N(x)$ and $\overline{N(x)}$ form the desired bipartition of the vertex set by picking $x$ uniformly at random. The expected number of edges in the neighborhood of $x$ is $3t/n$. The expected number of edges in $\overline{N(x)}$ is
\begin{align*}
\frac{1}{n}\sum_{x \in V(G)} e(\overline{N(x)}) & = \frac{1}{n}\sum_{(a,b) \in E(G)} \left(n-d(a)-d(b)+d(a,b)\right)\\ 
& = m+\frac{3t}{n}-\frac{1}{n}\sum_{a \in V(G)} d(a)^2 \leq m+\frac{3t}{n}-\frac{4m^2}{n^2},
\end{align*} 
where the first equality follows by double counting the number of triples $(x,a,b)$ of vertices where $(a,b)$ is an edge but $(x,a)$ and $(x,b)$ are not edges and the last inequality is by Cauchy--Schwarz. Thus, the expected number of edges in $N(x)$ and $\overline{N(x)}$ is at most $m+\frac{6t}{n}-\frac{4m^2}{n^2}$. Hence, there exists a choice of $x$ for which this random variable is at most the expected value.
\end{proof}

We use Lemma~\ref{firstlemma} to prove the following result, which gives conditions under which a graph contains a large induced bipartite subgraph. 

\begin{lem}\label{secondlemma}
Let $G$ be a graph with $n$ vertices, $m \geq n^2/4$ edges, $t \leq c^2n^3/24$ triangles and book number $b \leq \left(\frac{1}{2}-c\right)n$. Then $G$ contains an induced bipartite subgraph that contains all but at most $48t/cn^2$ vertices. 
\end{lem}

\begin{proof}
By Lemma \ref{firstlemma} and $m \geq n^2/4$, $G$ has a vertex partition $V(G)=A_0 \cup B_0$ such that all but at most $6t/n \leq c^2n^2/4$  edges are in $A_0 \times B_0$. We have $|A_0|,|B_0| \geq \left(1-c\right)n/2$, as otherwise the number of edges in $G$ is at most $|A_0||B_0|+c^2n^2/4 < n^2/4$, a contradiction. 

Let $A$ consist of all vertices $a \in A_0$ with more than $(b+|B_0|)/2$ neighbors in $B_0$. The set $A$ is independent, as otherwise we would have an edge in more than $b$ triangles. From each vertex $a \in A_0 \setminus A$, the number of missing edges to $B_0$ is at least $(|B_0|-b)/2 \geq cn/4$. Thus, we get at least $|A_0 \setminus A| \cdot cn/4$ missing edges from $A_0 \setminus A$ to $B_0$. We also have that the number of missing edges across $A_0 \times B_0$ is at most $|A_0||B_0|-(m-6t/n) \leq 6t/n$, so it follows that $|A_0 \setminus A| \leq (6t/n)/(cn/4) =24t/cn^2$. Similarly, letting $B$ consist of all vertices $b \in B_0$ with more than $(b+|A_0|)/2$ neighbors in $B_0$, we have that $B$ is independent and $|B_0 \setminus B| \leq 24t/cn^2$. Thus, $A \cup B$ induces a bipartite subgraph that contains all but at most $48t/cn^2$ vertices. 
\end{proof}

Our goal in this section is to prove Theorem \ref{thmbooknearquarter}. Since the proof is somewhat long, we first give an outline. Let $G$ be a graph on $n$ vertices with at least $n^2/4$ edges and book number at most $b$ (where $b<n/4$, but is not much smaller) which is not the balanced complete bipartite graph, but contains as few triangles as possible. We let $H$ be an induced bipartite subgraph of $G$ with the maximum number of vertices. Let $A$ and $B$ be the parts of $H$ and let $C$ be the remaining vertices, so that $A$, $B$ and $C$ form a vertex partition of $G$. We begin the proof by deriving some simple properties of the graph $G$. For instance, as there are not many triangles in $G$, we can use Lemma \ref{secondlemma} to deduce that $|C|$ is small. We can also conclude that $C$ is nonempty since $G$ has at least $n^2/4$ edges but is not complete bipartite. Moreover, by the choice of $H$, every vertex in $C$ has a neighbor in both $A$ and $B$. With a little more effort, we can even show that the degree of each vertex in $C$ is at least the maximum of $A$ and $B$. 

From this point on, we do not need to use the fact that each edge of $G$ is in at most $b$ triangles, just that a random edge from $E(A \cup B,C)$ is in expectation in at most $b$ triangles. We form a new graph $G_1$ on the same vertex set as $G$ by adding edges to $A \times B$ to make $A$ complete to $B$ and deleting the same number of edges from $E(A \cup B,C)$. We can do this so that in $G_1$ each vertex in $C$ has degree at most $b$ to $A$ and degree at most $b$ to $B$, the total number of triangles does not increase and a random edge in $(A \cup B) \times C$ is in expectation in at most $b$ triangles. We are not able to guarantee that $G_1$ has book number at most $b$, but tracking this related expectation is sufficient for our purposes. 


We now form another graph $G_2$ from $G_1$ by deleting edges in $E(C)$ and adding an equal number of edges to $(A \cup B) \times C$ so that each vertex in $C$ has degree $b$ to $A$ and degree $b$ to $B$. There are three types of triangle in $G_2$, those with exactly $i$ vertices in $C$ for $i=1,2,3$. It is easy to compute a lower bound on the number of type $1$ or $2$ triangles in $G_2$ and we show that this also gives a lower bound in $G_1$. Furthermore, the expected number of triangles containing a random edge of $E(G_2) \cap (A \cup B) \times C$ is at most the expected number of triangles containing a random edge of $E(G_1) \cap (A \cup B) \times C$. If $|C|<n-4b$, this expected number is larger than $b$, contradicting the fact that the corresponding expected number in $G$ is at most $b$. If $|C| \geq n-4b$, we find that the number of type $1$ or $2$ triangles in $G_2$ (and, hence, in $G$) is at least $b^2(n-4b)$, with equality only if $|C|=n-4b$. Furthermore, equality can occur only if $G = G_2$ and all triangles are of type $1$, so no edge in $C$ is in a triangle. But equality also implies that $|E(C)| \geq |C|^2/4$, so Mantel's theorem forces $C$ to induce a balanced complete bipartite graph. The parts of this partition determine two parts of the graph $S_{b,n}$, while the set of neighbors and nonneighbors of any vertex in $C$ partition each of $A$ and $B$ into two pieces, determining the remaining parts.

\vspace{3mm}

{\it Proof of Theorem \ref{thmbooknearquarter}.} Let $G$ be a graph on $n$ vertices with $m \geq n^2/4$ edges and book number at most $b=(1-\epsilon)n/4$, where $\epsilon \leq 1/500$, which is not the balanced complete bipartite graph, but for which the number $t$ of triangles is as small as possible. As the graph $S_{b,n}$ satisfies all of these conditions except possibly the last and has $b^2(n-4b)$ triangles, we may assume that $t \leq b^2(n-4b) \leq \epsilon n^3/16$. 

Let $H$ be the largest induced bipartite subgraph of $G$ and let $A$ and $B$ denote the parts of $H$ with $|A| \geq |B|$. Let $C=V(G)\setminus V(H)$. If a vertex in $C$ is not adjacent to some vertex in $A$, then we can add it to $A$ and get a larger induced bipartite subgraph of $G$, a contradiction. Since similar reasoning holds with $B$ in place of $A$, we have the following claim. 

{\bf Claim 1:} Every vertex in $C$ has a neighbor in both $A$ and $B$.  

By Lemma \ref{secondlemma} with $c=1/4$, we have the next claim. 

{\bf Claim 2:} $|C| \leq 192t/n^2 \leq 12\epsilon n$. 

If $|C|=0$, then $G$ is bipartite and, as the number of edges is at least $n^2/4$, $G$ has to be the balanced complete bipartite graph, a contradiction which yields the following claim. 

{\bf Claim 3:} The set $C$ is nonempty. 

We next observe that $G$ must have large minimum degree. 

{\bf Claim 4:} Every vertex in $B \cup C$ has degree at least $|A|$ and every vertex in $A$ has degree at least $|B|$. 

{\bf Proof:} Suppose that $v \in B \cup C$. If $d(v)<|A|$, we can delete all edges containing $v$ and then make $v$ complete to $A$. This operation increases the number of edges of $G$ and, as $v$ is not in any triangle in the new graph, does not increase $b(G)$ or $t(G)$. We can then delete an edge of the resulting graph which is in a triangle, obtaining a new graph $G'$ with at least $n^2/4$ edges which still has $b(G') \leq b$ but has fewer triangles than $G$. If $G'$ has zero triangles, then it is the complete balanced bipartite graph on an even number of vertices and the deleted edge would be in $n/2$ triangles, contradicting that the book number is at most $n/4$. Otherwise, $G'$ contradicts the choice of $G$ and the claim follows. The case where $v \in A$ follows similarly. \qed

As $A$ is an independent set with minimum degree at least $|B|$, each vertex $u \in A$ is adjacent to all but at most $|C| $ vertices in $B$. Similarly, every vertex in $B$ is adjacent to all but at most $|C|$ vertices in $A$.  We thus have the following claim. 

{\bf Claim 5:} Every vertex in $A$ (respectively, $B$) is adjacent to all but at most $|C|$ vertices in $B$ (respectively, $A$).

From Claims 1 and 5, we have the following claim, as otherwise $v$ is in an edge in more than $b$ triangles. 

{\bf Claim 6:} For every vertex $v \in C$, $d_A(v),d_B(v) \leq b+|C|$.

From the previous claim, for each vertex $v \in C$, we have 
$$d_A(v) = d(v)-d_C(v)-d_B(v) \geq |A|-|C|-(b+|C|) \geq \frac{n-|C|}{2}-|C|-(b+|C|)=\frac{n}{2}-b-\frac{5}{2}|C|.$$ 
Since the same bound clearly holds for $d_B(v)$, we have the following claim.

{\bf Claim 7:} For every vertex $v \in C$, $d_A(v),d_B(v) \geq \frac{n}{2}-b-\frac{5}{2}|C|$.

Let $D=D(G)=\max_{v \in C}(d_A(v),d_B(v))$ and $d = d(G)=\min_{v \in C}(d_A(v),d_B(v))$ so that $d \leq d_A(v),d_B(v) \leq D$ for all vertices $v \in C$. In general, for a graph parameter, we will usually not specify the graph if it is $G$, but we will if it is another graph, as we did in the proof of Claim 4. 

{\bf Claim 8:} $|C| > \frac{2}{3}\left(n-4b\right)$. 

{\bf Proof:} Suppose otherwise, that $|C| \leq \frac{2}{3}\left(n-4b\right)$.  If $D \leq b$, then the total number of edges of $G$ is at most 
\begin{eqnarray*} |A||B|+\sum_{v \in C} (d_A(v)+d_B(v))+{|C| \choose 2} & < & \left(\frac{n-|C|}{2}\right)^2+2b|C|+\frac{|C|^2}{2}
\\ & = & \frac{n^2}{4}+\frac{|C|}{2}\left(\frac{3}{2}|C|-(n-4b)\right) \\ & \leq & \frac{n^2}{4},\end{eqnarray*}
a contradiction. Thus, we must have $D > b$. Suppose $D=d_A(v)$ with $v \in C$ (the case $D=d_B(v)$ is handled in the same way). 
For each $u \in N_B(v)$, as the edge $(u,v)$ is in at most $b$ triangles, there must be at least $D-b$ missing edges from $u$ to $N_A(v)$. Hence, there are at least $(D-b)d_B(v)$ missing edges between $A$ and $B$. Then the number of edges in $G$ is at most 
\begin{eqnarray*} |A||B|-(D-b)d_B(v)+2D|C|+{|C| \choose 2} \! \!& < & \!\! \left(\frac{n-|C|}{2}\right)^2-(D-b)\left(\frac{n}{2}-b-\frac{5}{2}|C|\right)+2D|C|+\frac{|C|^2}{2} \\ \! \! & = &\!\! \frac{n^2}{4}+\frac{|C|}{2}\left(\frac{3}{2}|C|-(n-4b)\right)-(D-b)\left(\frac{n}{2}-b-\frac{9}{2}|C|\right)
\\\! \! & < &\! \! \frac{n^2}{4}+\frac{|C|}{2}\left(\frac{3}{2}|C|-(n-4b)\right) 
\\\!\!  & \leq & \! \!\frac{n^2}{4},
\end{eqnarray*} 
a contradiction.  The first inequality above uses Claim 7, while the second inequality uses $D-b>0$ and $\frac{n}{2}-b-\frac{9}{2}|C|>0$, which follows from $b \leq \frac{n}{4}$, Claim 2 and $\epsilon < 1/216$. 
\qed 

For $i \in \{0,1,2,3\}$, we say that a triangle in $G$ is of type $i$ if it contains exactly $i$ vertices from $C$. We let $t_i$ denote the number of triangles of type $i$. As there are no triangles in $H=G[A \cup B]$, we have $t_0=0$. Let $t'=t_1+t_2$ be the number of triangles of type 1 or 2.

Let $\bar b(G)$ denote the expected number of triangles containing a random edge in $E(A \cup B,C)$. That is, $\bar b(G)=2t'(G)/e_G(A \cup B,C)$. 

{\bf Claim 9:} There is a graph $G_1$ with $V(G_1)=V(G)$ and $e(G_1)=e(G)$ such that $G_1$ induces a complete bipartite graph on $A \cup B$ with parts $A$ and $B$, $D(G_1) \leq b$, $d(G_1) \geq \frac{n}{4}-\frac{5}{2}|C|$, $t(G_1) \leq t(G)$, $t'(G_1) \leq t'(G)$ and $\bar b(G_1) \leq \bar b(G)$. Moreover, if $G_1 \not =G$, then $t(G_1)<t(G)$.

{\bf Proof:} 
Suppose there are $s$ missing edges between $A$ and $B$ in $G$. Consider adding all $s$ missing edges between $A$ and $B$ (so $A$ is now complete to $B$) and then deleting $s$ edges between $C$ and $A \cup B$, deleting them one at a time from a vertex in $C$ of largest degree to $A$ or $B$ to obtain a new graph $G_1$. To see that this process is possible, note that each vertex $v \in B$ has degree at least $|A|$ by Claim 4 and so has at least as many neighbors in $C$ as it has nonneighbors in $A$. Note, by construction, that $V(G_1) = V(G)$ and $e(G_1)=e(G)$.

If $D(G)>b$ and $v$ is a vertex with $d_A(v)=D(G)$ (the case $d_B(v)=D(G)$ is handled in the same way), then, in the graph $G$, for each $u \in N_B(v)$, the edge $(u,v)$ is in at most $b$ triangles, so $u$ has at least $D(G)-b$ missing edges to $A$. Thus, by Claim 7, 
$$s \geq d_B(v)(D(G)-b) \geq \left(\frac{n}{2}-b-\frac{5}{2}|C|\right)(D(G)-b) \geq \left(\frac{n}{4}-\frac{5}{2}|C|\right)(D(G)-b)  \geq 2|C|(D(G)-b),$$
where the last inequality follows from Claim 2 and $\epsilon < 1/216$. The final expression is an upper bound on the number of edges that must be deleted between $C$ and $A \cup B$ to guarantee $\max_{v \in C}(d_A(v),d_B(v)) \leq b$. We thus have $D(G_1) \leq b$ in this case. If $D(G) \leq b$, then, since we only deleted edges between $C$ and $A \cup B$ to make $G_1$, $D(G_1) \leq D(G) \leq b$. Hence, in either case, we have $D(G_1) \leq b$. 

Observe that if $D(G_1)>d(G_1)+1$, then we must have $d(G_1)=d(G)$ as if, say, $d_A(v)=d(G)$, we would never delete an edge from $v$ to $A$ in the process of obtaining $G_1$. In this case, we have, by Claim 7, that $d(G_1)=d(G) \geq \frac{n}{2}-b-\frac{5}{2}|C| \geq \frac{n}{4}-\frac{5}{2}|C|$. Otherwise, we have that the degrees $d_A(v)$, $d_B(v)$ in $G_1$ are all simply the average degree rounded up or down. The number of edges of $G_1$ between $C$ and $A \cup B$ satisfies
$$e_{G_1}(A \cup B,C) \geq \frac{n^2}{4}-|A||B|-{|C| \choose 2} \geq \frac{n^2}{4}-\left(\frac{n-|C|}{2}\right)^2-\frac{|C|^2}{2}=\frac{|C|n}{2}-\frac{3}{4}|C|^2.$$
So the average value of $d_X(v)$ over all $2|C|$ choices of $v \in C$ and $X \in \{A,B\}$ is at least $\frac{n}{4}-\frac{3}{8}|C|$. 
Hence, 
\begin{equation}\label{dG1} 
d(G_1) \geq \min\left(\frac{n}{4}-\frac{5}{2}|C|,\frac{n}{4}-\frac{3}{8}|C|-1\right) \geq \frac{n}{4}-\frac{5}{2}|C|.
\end{equation}

Since each of the $s$ edges added between $A$ and $B$ is in at most $|C|$ triangles, in total this process added at most $s|C|$ triangles. Once these edges have been added and the graph between $A$ and $B$ is complete bipartite, we remove the $s$ edges from between $C$ and $A \cup B$. Since each such edge is contained in at least $d(G_1)$ triangles, we remove at least $s d(G_1)$ triangles in total. Hence, 
$$t'(G_1)-t'(G) \leq s(|C|-d(G_1)).$$ 
As $n \geq 14 \cdot 12 \epsilon n > 14|C|$ for $\epsilon < 1/168$, it follows from (\ref{dG1}) that $t'(G_1) \leq t'(G)$. As no edges in $C$ are added or deleted in obtaining $G_1$ from $G$, we have $t_3(G_1)=t_3(G)$ and, hence, $t(G_1) \leq t(G)$. Moreover, if $s \not =0$, then $t'(G_1)<t'(G)$ and, hence, $t(G_1)<t(G)$.

Finally, we check that $\bar b(G_1) \leq \bar b(G)$. This is equivalent to showing that
$$\frac{2t'(G_1)}{e_{G_1}(A \cup B,C)} \leq \frac{2t'(G)}{e(A \cup B,C)}$$
and, as $e_{G_1}(A \cup B,C)=e(A \cup B,C)-s$, this is equivalent to showing that
$$\left(t'(G)-t'(G_1)\right)e(A \cup B,C) \geq st'(G).$$
From the bound $t'(G)-t'(G_1) \geq s(d(G_1)-|C|)$, this would follow if we could show that
$$(d(G_1)-|C|)e(A \cup B,C) \geq t'(G).$$
Each edge in $E(A \cup B,C)$ is in at most $b$ triangles in $G$ and each type $1$ or $2$ triangle has exactly two such edges, so $t'(G) \leq e(A \cup B,C)b/2$. Hence, it suffices to show that $d(G_1)-|C|\geq b/2$, which follows from (\ref{dG1}), $|C| \leq 12\epsilon n$, $b \leq n/4$ and $\epsilon \leq 1/336$.  This completes the proof of Claim 9. \qed

{\bf Claim 10:} $|C| \leq \frac{1}{1-240\epsilon}(n-4b) \leq 2(n-4b)=2\epsilon n$. 
 
{\bf Proof:} Suppose, for the sake of contradiction, that $|C|>\frac{1}{1-240\epsilon}(n-4b)$. Each vertex $v \in C$ is in $d_A(v,G_1)d_B(v,G_1)$ type $1$ triangles in $G_1$. As $d_A(v,G_1)d_B(v,G_1)\geq d(G_1)^2 \geq \left(\frac{n}{4}-\frac{5}{2}|C|\right)^2$ by Claim 9, the number of triangles in $G_1$ is at least $$|C|\left(\frac{n}{4}-\frac{5}{2}|C|\right)^2 \geq \left(\frac{n}{4}\right)^2|C|\left(1-\frac{20|C|}{n}\right).$$
This last expression is an increasing function of $|C|$ for $|C|\leq \frac{n}{40}$ (which holds since $|C| \leq 12\epsilon n$ and $\epsilon \leq 1/480$). Hence, as $b \leq n/4$ and $\frac{1}{1-240\epsilon}(n-4b) < |C| \leq 12\epsilon n$, we have that the number of triangles in $G_1$ (and, hence, $G$) is larger than $b^2(n-4b)$, a contradiction. \qed

For any graph $G'$ on $V(G)$ for which $A \cup B$ induces a complete bipartite graph, the number of triangles containing an edge $(u,v) \in E_{G'}(A,C)$ is 
$$d_B(u,v,G')+d_{C}(u,v,G')=d_B(v,G')+d_C(u,v,G') \geq d_B(v,G')+d_C(u,G')+d_C(v,G')-|C|.$$
Similarly, if $(u,v) \in E_{G'}(B,C)$, then the number of triangles in $G'$ containing the edge $(u,v)$ is at least $d_A(v,G')+d_C(u,G')+d_C(v,G')-|C|$. Summing over all edges in $E_{G'}(A \cup B,C)$ and using the fact that each type 1 or 2 triangle contains exactly two such edges, the number of type 1 or 2 triangles in $G'$ is at least $\tilde{t}(G')$, defined by 
$$2\tilde{t}(G'):=-|C|e_{G'}(A \cup B,C)+\sum_{v\in C} \left(2d_A(v,G')d_B(v,G')+d_C(v,G')d_{A \cup B}(v,G')\right)+\sum_{u \in A \cup B}d_C(u,G')^2.$$
To see this, note, for example, that each term of the form $d_C(u, G')$ appears $d_C(u, G')$ times, once for each edge $(u, v) \in E_{G'}(A \cup B, C)$.

{\bf Claim 11:} There is a graph $G_2$ obtained from $G_1$ by deleting some edges with both vertices in $C$ and adding an equal number of edges to $(A \cup B) \times C$ such that $d_A(v,G_2)=d_B(v,G_2)=b$ for all $v \in C$, $\tilde{t}(G_2) \leq \tilde{t}(G_1)$ and $e_{G_2}(A \cup B,C) \geq e_{G_1}(A \cup B,C)$. Moreover, if 
$G_2 \not =G_1$, then $\tilde{t}(G_2)<\tilde{t}(G_1)$.

{\bf Proof:} As $d_A(v,G_1),d_B(v,G_1) \leq b$, we can arbitrarily delete edges from $C$ (as long as there are edges) and add an equal number of edges to $(A \cup B) \times C$  to obtain the graph $G_2$ with $d_A(v,G_2)=d_B(v,G_2)=b$. This is possible because, by Claim 10 and $b = (1 - \epsilon)n/4$, the number of edges we would get, not including those in $C$, is 
$$|A||B|+|C|2b \leq \left(\frac{n-|C|}{2}\right)^2+|C|2b = \frac{n^2}{4}+\frac{|C|^2}{4}-\frac{\epsilon}{2}|C|n \leq \frac{n^2}{4},$$
leaving enough room for a nonnegative number of edges in $C$. Note that, by construction, $G_2$ has at least as many edges across $(A \cup B) \times C$ as $G_1$. 

Let $G'$ be a graph obtained at some stage of the process of transforming $G_1$ into $G_2$. If we delete an edge $(v,v')$ from $G'$ with $v,v'\in C$ to obtain $G''$, then it decreases the value of $2\tilde{t}(G')$ by $d_{A \cup B}(v,G')+d_{A \cup B}(v',G') \geq 4\left( \frac{n}{4}-\frac{5}{2}|C|\right)=n-10|C|$, where the inequality is by the lower bound on $d(G_1)$ from Claim 9. If we add an edge $(u,v) \in (A \cup B) \times C$ to this graph (with, say, $u \in A$), it increases the value of $2\tilde{t}(G'')$ by $$-|C|+ 2 d_B(v,G'')+d_C(v,G'')+2d_C(u,G'')+1 \leq 2|C|+1+ 2b,$$ where the last inequality uses $d_C(v,G''),d_C(u,G'') \leq |C|$ and $d_B(v,G'') \leq b$. Hence, in deleting an edge with both vertices in $C$ and adding an edge in $(A \cup B) \times C$, we decreased the value of $\tilde{t}(G')$ by at least $n-10|C|-(2|C|+1+2b) \geq \frac{n}{2}-13|C| \geq \frac{n}{2}-156\epsilon n > 0$, where we used Claim 2. Thus, in the process of going from $G_1$ to $G_2$, $\tilde{t}$ decreases at each step, so 
 $\tilde{t}(G_2) \leq \tilde{t}(G_1)$, with equality only if $G_2 = G_1$. \qed 

We have 
\begin{eqnarray*} 2\tilde{t}(G_2) & = & -2|C|^2b+2|C|b^2+4be_{G_2}(C)+\sum_{u \in A \cup B} d_C(u,G_2)^2 \\ & \geq & -2|C|^2b+2|C|b^2+4b\left(n^2/4-|A||B|-2b|C|\right)+4|C|^2b^2/|A \cup B| \\ & \geq & 
-2|C|^2b+2|C|b^2+4b\left(n^2/4-\left((n-|C|)/2\right)^2 -2b|C|\right)+4|C|^2b^2/|A \cup B|
\\ & = & -3|C|^2b-6|C|b^2+2|C|bn+4|C|^2b^2/(n-|C|),\end{eqnarray*} 
where, in the first inequality, we used the Cauchy--Schwarz inequality and $\sum_{u \in A \cup B} d_C(u,G_2) = 2 |C| b$. The last expression, as a function of $|C|$, is increasing for $b$ in the range of interest and $|C|$ in the range determined by Claims 8 and 10, which can be seen by taking the derivative with respect to $|C|$. Using this fact, we may evaluate this expression at $|C|=n-4b$ to conclude that 
$$b^2(n-4b) \leq \tilde{t}(G_2) \leq \tilde{t}(G_1) \leq t(G_1) \leq t(G)$$ 
for $|C| \geq n-4b$. Furthermore, the only way we could get equality in the above bound is if $|C|=n-4b$, $|A|=|B|=2b$ and if we moved no edges in getting $G_1$ from $G$ and $G_2$ from $G_1$, so that $G_2$ and $G$ are the same. Therefore, in $G$, $A$ is complete to $B$ and $d_A(v)=d_B(v)=b$ for each vertex $v \in C$. Hence, as each vertex in $C$ is in $b^2$ type $1$ triangles, the number of triangles of type $1$ in $G$ is $b^2(n-4b)$  so there are no type $2$ or $3$ triangles in $G$. In particular, no edge in $C$ belongs to a triangle. On the other hand,
$$e(C) \geq \frac{n^2}{4} - |A||B| - 2b|C| \geq \frac{n^2}{4} - \left(\frac{n - |C|}{2}\right)^2 - 2b|C| = -\frac{|C|^2}{4} + \frac{\epsilon}{2} |C| n = \frac{|C|^2}{4},$$
where, in the last inequality, we used that $|C| = n - 4 b = \epsilon n$. As $C$ has at least $|C|^2/4$ edges but induces a triangle-free graph, Mantel's theorem implies that $|C|$ is even (which is equivalent to $n$ being even) and $C$ induces a balanced complete bipartite graph with parts $C_1$, $C_2$ of equal size. As no edge in $C$ is in a triangle with a vertex in $A$ or $B$ and yet $d_A(v)=b=|A|/2$ and $d_B(v)=b=|B|/2$ for each $v \in C$, we have equitable partitions $A=A_1 \cup A_2$ and $B=B_1 \cup B_2$ such that $C_1$ is complete to $A_1 \cup B_1$, $C_2$ is complete to $A_2 \cup B_2$ and there are no other edges between $A \cup B$ and $C$. It is now easy to check that $G$ is the graph  $S_{n,b}$ with parts $A_1,B_1,C_1,B_2,A_2,C_2$.

It remains to check the case $|C|<n-4b$. We will show that there is an edge in more than $b$ triangles, a contradiction. Indeed, 
\begin{eqnarray*} b(G) & \geq & \bar b(G) \geq \bar b(G_1) = 2 t'(G_1)/e_{G_1}(A \cup B,C) \geq 2 \tilde{t}(G_1)/e_{G_1}(A \cup B,C)\\ & \geq & 2 \tilde{t}(G_2)/e_{G_1}(A \cup B,C) \geq 2 \tilde{t}(G_2)/e_{G_2}(A \cup B,C).\end{eqnarray*}
This last expression is at least 
$$\frac{1}{2b|C|}\left( -3|C|^2b-6|C|b^2+2|C|bn+4|C|^2b^2/(n-|C|) \right) = -\frac{3}{2}|C|-3b+n+2|C|b/(n-|C|).$$
In the range of interest, this function is strictly decreasing in $|C|$. Given that we are assuming that $|C|<n-4b$, if we evaluate the above expression at $|C|=n-4b$, we get $b$ and, hence, $b(G)$ is greater than this value, a contradiction. This completes the proof of Theorem \ref{thmbooknearquarter}. \qed

\section{Proof of Theorem~\ref{thmbooknearsixth}} \label{sec:sixth}

The main result of this section, which easily implies Theorem~\ref{thmbooknearsixth}, is as follows.

\begin{thm}
If $\epsilon > 0$ is sufficiently small, then every graph on $n$ vertices with at least $n^2/4$ edges and book number at most $\left( \frac16 + \epsilon^3 \right) n$ which is not the balanced complete bipartite graph has at least $\left( \frac{1}{108} - O( \epsilon ) \right) n^3$ triangles.
\end{thm}

\begin{proof}
Suppose that $G$ is a graph satisfying the assumptions of the theorem. As with Theorem~\ref{thmbooknearquarter}, we will prove the result through a sequence of claims.

{\bf Claim A:}
$G$ is approximately regular, in that $| \{ v : \left| d(v) - \frac{n}{2} \right| \ge \epsilon n \} | \le \epsilon n$.

{\bf Proof:}
Let $b = b(G)$, $t(G)$ be the number of triangles in $G$ and $m$ be the number of edges. We use the following inequality, proved by Bollob\'as and Nikiforov \cite{BN05} (see Equation (8)),
\[ ( 6 b - n ) t(G) \ge b \left( \sum_v d(v)^2 - nm \right). \]
Since  $\sum_v d(v)^2 = \sum_v \left( d(v) - \frac{n}{2} \right)^2 + 2mn - \frac{n^3}{4}$, we have
\[ ( 6 b - n ) t(G) \ge b \left( \sum_v \left( d(v) - \frac{n}{2} \right)^2 + nm - \frac{n^3}{4} \right) \geq b \sum_v \left( d(v) - \frac{n}{2} \right)^2. \]
As the right-hand side is non-negative and $t(G) > 0$ (since $G$ is not the balanced complete bipartite graph), it follows that $6 b - n \ge 0$.  Using the simple bound $t(G) \le \frac13 b m$, we find that 
$$6 b - n \ge \frac{3}{m} \sum_v \left( d(v) - \frac{n}{2} \right)^2 \ge \frac{3}{m} | \{ v : | d(v) - \tfrac{n}{2} | \ge \epsilon n \} | \epsilon^2 n^2.$$
Suppose now that $| \{ v : \left| d(v) - \frac{n}{2} \right| \ge \epsilon n \} | > \epsilon n$.  Substituting this in and using $m \leq n^2/2$ yields $6 b - n > \frac{3 \epsilon^3 n^3}{m} \ge 6 \epsilon^3 n$ and, hence, $b > \left( \frac16 + \epsilon^3 \right)n$, a contradiction.
\qed

Now remove any vertices of degree less than $\left( \frac12 - \epsilon \right)n$ from $G$.  By Claim A, this gives a new graph $G'$ on $n' \geq (1 - \epsilon)n$ vertices.
Since $G$ had at least $n^2/4$ edges and we removed at most $(n - n')(\frac{1}{2}-\epsilon) n$ edges, $G'$ also has at least
$(n')^2/4$ edges. The minimum degree of $G'$ is at least $(\frac{1}{2}-2\epsilon)n \geq (\frac{1}{2}-2\epsilon)n'$ and $b(G') \leq ( \frac16 + \epsilon^3) n \leq ( \frac16 + \frac{\epsilon}{5}) n'$. For simplicity, we shall again call this smaller graph $G$ and suppose that it has $n$ vertices. Furthermore, increasing $\epsilon$ by at most a factor $2$, we have that the minimum degree of $G$ is at least $(\frac{1}{2}-\epsilon)n$ and $b(G) \leq ( \frac16 + \frac{\epsilon}{10}) n$.  The additional error introduced by increasing  $\epsilon$ is easily covered by the $O(\epsilon)$ term in our bound on the number of triangles.

Given any vertex $v \in G$, we will use the shorthand $N_v$ for the neighbors of $v$ and $M_v$ for the set of nonneighbors (including $v$).  Note that we have $|N_v| \ge \left( \frac12 - \epsilon \right)n$ and $|M_v| \le \left( \frac12 + \epsilon \right) n$ for all $v$.  Clearly, for any $v \in G$ and $x \in N_v$, we have $d_{N_v}(x) \le b(G) \le \left( \frac16 + \frac{\epsilon}{10} \right)n$.

{\bf Claim B:}
Given $v \in G$ and $x \in N_v$, if $d_{N_v}(x) \ne 0$, then $\left( \frac16 - 4\epsilon \right)n \le d_{N_v}(x) \le \left( \frac16 + \frac{\epsilon}{10} \right) n$.

{\bf Proof:}
Let $y \in N_v$ be a neighbour of $x$.  Note that $x$ and $y$ both have degree at least $\left( \frac12 - \epsilon \right)n$.  Thus, the number of common neighbors in $M_v$ is at least 

\begin{eqnarray}
\label{claimB}
d_{M_v}(x,y) &\ge& d_{M_v}(x) + d_{M_v}(y) - |N_{M_v}(x) \cup N_{M_v}(y)| \ge d_{M_v}(x) + d_{M_v}(y) -|M_v|\nonumber\\ 
&\ge& \left( \frac12 - 3 \epsilon \right) n - d_{N_v}(x) - d_{N_v}(y),
\end{eqnarray}

where we used that $d_{M_v}(x) = d(x) - d_{N_v}(x) \geq \left(\frac{1}{2} - \epsilon\right) n - d_{N_v}(x)$ and $|M_v| \le \left( \frac12 + \epsilon \right) n$. Using the bounds $d_{M_v}(x,y) \le d(x,y) \le b(G) \le \left( \frac16 + \frac{\epsilon}{10} \right)n$ and $d_{N_v}(y) \le b(G) \le \left( \frac16 + \frac{\epsilon}{10} \right) n$, we deduce that $d_{N_v}(x) \ge \left( \frac16 - 3 \epsilon - \frac{\epsilon}{5} \right) n \ge \left( \frac16 - 4 \epsilon \right)n$.
\qed

This proof also shows that if $x$ and $y$ are neighbors in $N_v$, then they have at least $\left( \frac16 - 4 \epsilon \right)n$ common neighbors in $M_v$. Therefore, 
they can have at most $O( \epsilon n)$ common neighbors in $N_v$. Note also that we must have $d(v) = |N_v| \le \left( \frac12 + O( \epsilon ) \right) n$, since a smaller bound on the size of $|M_v|$ would force $d_{M_v}(x,y) > \left( \frac16 + \frac{\epsilon}{10} \right) n$.

{\bf Claim C:}
Suppose $e(N_v) > 0$.  Then $e(N_v) \ge \left( \frac{1}{36} - O( \epsilon ) \right)n^2$ and there are at least $\left( \frac{1}{216} - O( \epsilon ) \right)n^3$ triangles with two vertices in $N_v$ and one in $M_v$.

{\bf Proof:}
Suppose $x \sim y$ in $N_v$.  Then $d_{N_v}(x) > 0$, so $d_{N_v}(x) \ge \left( \frac16 - 4 \epsilon \right)n$ and similarly for $y$.  Moreover, since $x$ and $y$ have at most $O( \epsilon n)$ common neighbors in $N_v$, the neighbors of $x$ and the neighbors of $y$ in $N_v$ give at least $\left( \frac13 - O( \epsilon ) \right)n$ vertices of positive degree in $G[N_v]$, each of which has degree at least $\left( \frac16 - 4 \epsilon \right)n$.  Thus, $e(N_v) \ge \frac12 \left( \frac13 - O( \epsilon ) \right) \left( \frac16 - 4 \epsilon \right) n^2 = \left( \frac{1}{36} - O( \epsilon ) \right) n^2$.  Moreover, each of these edges must have at least $\left( \frac16 - O( \epsilon ) \right) n$ common neighbors in $M_v$, giving $\left( \frac{1}{36} - O( \epsilon ) \right) \left( \frac16 - O( \epsilon ) \right) n^3 = \left( \frac{1}{216} - O( \epsilon ) \right) n^3$ triangles with two vertices in $N_v$ and one in $M_v$.
\qed

{\bf Claim D:}
For every $v \in G$, $e(N_v) = e(M_v) + O( \epsilon n^2)$.

{\bf Proof:}
We have $\sum_{x \in N_v} d(x) = 2e(N_v) + e(N_v,M_v)$ and $\sum_{x \in M_v} d(x) = 2e(M_v) + e(N_v, M_v)$.  Since the graph is almost regular by Claim A and $|N_v|, |M_v| = \left( \frac12 + O( \epsilon ) \right)n$, it follows that the two sums are approximately equal, that is, $e(N_v) = e(M_v) + O( \epsilon  n^2)$.
\qed

{\bf Claim E:}
For every $v \in G$, there is some $w \in N_v$ with $d_{N_v}(w) = 0$.

{\bf Proof:}
Suppose on the contrary that $d_{N_v}(w) > 0$ for all $w \in N_v$.  Then, by $|N_v| \geq \left( \frac12 - \epsilon \right)n$ and Claim B,  
\[ e(N_v) \ge \frac12 \left( \frac12 - \epsilon \right) \left( \frac16 - 4 	\epsilon \right)n^2 = \left( \frac{1}{24} - O( \epsilon ) \right) n^2. \]
Each of these edges extends to at least $\left( \frac16 - O( \epsilon ) \right)n$ triangles with a vertex in $M_v$, giving at least $\left( \frac{1}{144} - O(\epsilon) \right)n^3$ such triangles. Moreover, by Claim D, we have $e(M_v) \ge \left( \frac{1}{24} - O( \epsilon ) \right) n^2$.  

Now consider any edge $x \sim y$ in $N_v$ and count the number of triangles containing $x$ or $y$ with two vertices in $M_v$.  By Claim B, we have $|N_{M_v}(x)| \leq \left(\frac{1}{3} + O(\epsilon)\right) n$. Together with $d_{M_v}(x, y) \ge \left( \frac16 - 4 \epsilon \right)n$, this implies that $| N_{M_v}(x) \setminus N_{M_v}(y) | \le \left(\frac16 + O(\epsilon) \right)n$ and similarly for $| N_{M_v}(y) \setminus N_{M_v}(x)|$.  
The proof of Claim B also implies that $|M_v \setminus (N_{M_v}(x) \cup N_{M_v}(y))|=O(\epsilon n)$, otherwise, from inequality (\ref{claimB}), we  have that $d_{M_v}(x,y)$ is too big, a contradiction.
Therefore, there are at most $\left( \frac{1}{36} + O(\epsilon) \right)n^2$ edges in $M_v$ that do not form a triangle with $x$ or $y$.  This leaves at least $\left( \frac{1}{72} - O(\epsilon) \right)n^2$ edges that do form a triangle with at least one of $x$ or $y$.  Summing these up for every edge in $N_v$ gives $\left( \frac{1}{1728} - O( \epsilon ) \right)n^4$. By Claim B, each vertex $x$ has degree at most $\left(\frac16 + \frac{\epsilon}{10} \right) n$ in $N_v$, so this gives an upper bound on the number of times any triangle can be counted.  Hence, the total number of such triangles (one vertex in $N_v$, two in $M_v$) is at least $\left( \frac{1}{288} - O(\epsilon) \right)n^3$.  This implies that the total number of triangles in $G$ is at least $\left( \frac{1}{96} - O( \epsilon ) \right) n^3$, which is too large.
\qed

We may now complete the proof. Since $e(G) \geq n^2/4$ and $G$ is not the balanced complete bipartite graph, $G$ must contain a triangle. Let $v$ be a vertex of this triangle, so that $e(N_v) > 0$. By Claim C, we have $e(N_v) \ge \left( \frac{1}{36} - O( \epsilon ) \right)n^2$ and at least $\left( \frac{1}{216} - O( \epsilon ) \right) n^3$ triangles with two vertices in $N_v$.  Moreover, by Claim E, there is some $w \in N_v$ with $d_{N_v}(w) = 0$, which implies that $N_w \subset M_v$.  Note now that $|M_v \setminus N_w | = O( \epsilon  n)$.  Since $e(M_v) = e(N_v) + O(\epsilon n^2)$ by Claim D, it follows that $e(N_w) \ge e(N_v) - O( \epsilon n^2) > 0$.  Hence, again by Claim C, there are at least $\left( \frac{1}{216} - O( \epsilon ) \right)n^3$ triangles with two vertices in $N_w$.  Since $N_w \cap N_v = \emptyset$, these triangles are distinct from those above, which gives $\left( \frac{1}{108} - O( \epsilon ) \right)n^3$ triangles in total.
\end{proof}

If $\epsilon = 0$ and equality holds throughout the argument above, consider a vertex $v$ which is contained in a triangle and a vertex $x \in N_v$ with $d_{N_v}(x) \neq 0$. 
Then, by Claim A, the graph is $n/2$-regular and, by Claim B, we have $d_{N_v}(x) = n/6$. Note, moreover, that $N_v$ is triangle-free by the comments after Claim B, which implies that $N(v, x)$ is an independent set. Similarly, for any $y \in N(v, x)$, $d_{N_v}(y) = n/6$ and $N(v, y)$ must be an independent set. We now split $N_v$ into three parts, each with $n/6$ vertices, namely, $N(v, x)$, $N(v, y)$ and the remainder, which we label $R_v$. 
By the proof of Claim C, we see that if $e(N_v)>n^2/36$, then there are more than $n^3/216$ triangles with two vertices in $N_v$ and one in $M_v$. Since, by Claim E, there is a vertex $w \in N_v$ with no neighbors in $N_v$, we have that
$N_w=M_v$ and, hence, there are at least $n^3/216$ triangles with two vertices in $M_v$ and one in $N_v$. So altogether there are more than $n^3/108$ triangles, a contradiction. This implies that $e(N_v)=n^2/36$ and, therefore, there are exactly $n/3$ vertices in $N_v$
with degree $n/6$ in $N_v$. Since the neighbors of $x$ and $y$ must all have positive degree in $N_v$ (which by the above discussion should be $n/6$), we conclude that the vertices in $R_v$ have no neighbors in $N_v$, while there must be a complete bipartite graph between $N(v,x)$ and $N(v,y)$. 

Picking now any vertex $u \in R_v$, we see that its neighborhood must be $M_v$, the complement of $N_v$. By the same argument as above, the induced graph on $M_v = N_u$ must consist of a balanced complete bipartite graph between two parts $N(u, x'), N(u, y')$, each with $n/6$ vertices, and a set $R_u$ of $n/6$ vertices with no neighbors in $M_v$, each of which must then be complete to $N_v$. Since there are $n^3/216$ triangles between $R_u, N(v, x)$ and $N(v, y)$ and a similar number between $R_v, N(u, x')$ and $N(u, y')$, we see that there are no more triangles, so any vertex in $N(v, x) \cup N(v, y)$ can only have neighbors in one of $N(u, x')$ or $N(u, y')$ and vice versa. Putting all this together, we see that equality holds only if the graph is the blow-up of a $3$-prism with $n/6$ vertices in each part, as claimed.

\vspace{-2mm}

\section{Concluding remarks}

The most obvious question that we have left open is Conjecture~\ref{mainconj}. Our results only establish this conjecture when $b = n/6$ or when $0.2495 n \leq b < n/4$, so much more remains to be done. In the first instance, it might be interesting to show that there is some $\epsilon > 0$ such that the conjecture holds for all $n/6 \leq b \leq \left(\frac{1}{6} + \epsilon\right) n$. 

There are of course many natural variants of Mubayi's question: how does the tradeoff between triangles and books change if we assume there are at least $\alpha n^2$ edges for some $1/4 < \alpha < 1/2$? what happens for larger cliques? what about hypergraphs? But the question also points to a more general metaquestion, of how  the local and global counts for substructures play off against one another. There are many contexts besides graphs in which such questions can be asked.

\vspace{3mm}
{\bf Acknowledgements.} We are grateful to Shagnik Das and Nina Kam\v cev for helpful conversations and are particularly indebted to Shagnik for writing up an early draft of Section~\ref{sec:sixth}. We would also like to thank the anonymous referees for their detailed and insightful reports.


\begin{thebibliography}{}

\bibitem{Bol}
B. Bollob\'as, On complete subgraphs of different orders, {\it Math. Proc.
Camb. Phil. Soc.} {\bf 79} (1976), 19--24.


\bibitem{BN05}
B. Bollob\'as and V. Nikiforov, Books in graphs, {\it European J. Combin.} {\bf 26} (2005), 259--270. 


\bibitem{Er62} P. Erd\H{o}s, On a theorem of Rademacher--Tur\'an, {\it Illinois J. Math.} {\bf 6} (1962), 122--127. %

\bibitem{Er62b}
P. Erd\H{o}s,  On the number of complete subgraphs contained in certain graphs, {\it Magyar Tud. Akad. Mat. Kutat\'o Int. K\"ozl.} {\bf 7} (1962), 459--474.

\bibitem{EFPS} P. Erd\H{o}s, R. Faudree, J. Pach and J. Spencer, How to make a graph bipartite, 
{\it J. Combin. Theory Ser. B} {\bf 45} (1988), 86--98. 

\bibitem{LPS}
H. Liu, O. Pikhurko and K. Staden,
The exact minimum number of triangles in graphs of given order and size,
preprint available at arXiv:1712.00633 [math.CO].

\bibitem{Lov-Sim} L. Lov\'asz and M. Simonovits, On the number of complete subgraphs of
a graph II, in {\it Studies in pure mathematics}, 459--495, Birkha\"user, Basel, 1983.

\bibitem{KN79}
N. Khad\v{z}iivanov and V. Nikiforov, Solution of a problem of P. Erd\H{o}s about the maximum number of triangles with a common edge in a graph, {\it C. R. Acad. Bulgare Sci.} {\bf 32} (1979), 1315--1318 (in Russian). %

\bibitem{Man}
W. Mantel, Problem 28, {\it Wisk. Opg.} {\bf 10} (1907), 60--61.

\bibitem{Mubayi} D. Mubayi, Books versus triangles, {\it J. Graph Theory} {\bf 70} (2012), 171--179. %

\bibitem{Nikiforov}
V. Nikiforov, The number of cliques in graphs of given order and size,
 {\it Trans. Amer. Math. Soc.} {\bf 363} (2011), 1599--1618.


\bibitem{Razborov}
A. Razborov, On the minimal density of triangles in graphs, {\it Combin.
Probab. Comput.}  {\bf 17} (2008), 603--618.

\bibitem{Reiher}
C. Reiher, The clique density theorem, {\it Ann. of Math.} {\bf 184} (2016), 683--707.


\end{thebibliography}
\end{document}